\documentclass[12pt,a4paper]{article}    
\usepackage{amsmath,amsfonts,amssymb,amsthm,amscd}
\usepackage[english]{babel}
\usepackage{indentfirst}
\usepackage{a4wide}
\usepackage{mathrsfs}  
\usepackage[utf8]{inputenc}

\numberwithin{equation}{section}

\newtheorem{theorem}{Theorem}[section]

\newtheorem{proposition}[theorem]{Proposition}
\newtheorem{lemma}[theorem]{Lemma}
\newtheorem{remark}[theorem]{Remark}

\newcommand\nl[2]{\|#2\|_{L^{#1}}}
\newcommand\nlo[2]{\|#2\|_{L^{#1}(\Om)}}
\newcommand\nh[2]{\|#2\|_{H^{#1}}}
\newcommand\nhu[1]{\|#1\|_{H^1}}
\newcommand\nhua[1]{\|#1\|_{H^1_\al}}
\newcommand\sob[3]{\|#3\|_{W^{#1,#2}}}

\renewcommand{\leq}{\leqslant}
\renewcommand{\geq}{\geqslant}

\DeclareMathOperator{\dive}{div}
\DeclareMathOperator{\curl}{curl}
\def\ep{\varepsilon}
\def\al{\alpha}
\def\om{\omega}
\def\Om{\Omega}
\def\dt{\partial_{t}}
\def\la{\lambda}
\def\wt{\widetilde}
\def\A{\mathbb{A}}
\def\AA{\mathbb{A}_2}
\def\C{\mathbb{C}}
\def\Ga{\Gamma}
\def\ualk{{u^{\al_k}}}

\def\ual{{u^{\al}}}
\def\uall{u^{\al}}
\def\ualz{{u^\al_0}}

\def\tu{\widetilde{u}}
\def\txi{{Y}^\al_i}
\def\qb{{\overline{q}}}
\def\ub{{\overline{u}}}

\def\qbb{{\check{q}}}
\def\ubb{{\check{u}}}
\def\vbb{{\check{v}}}
\def\rha{\rightharpoonup}
\def\hra{\hookrightarrow}
\def\omb{\overline\om}
\def\T{\mathbb{T}}
\def\S{\mathbb{S}}
\def\R{\mathbb{R}}
\def\PP{\mathbb{P}}

\begin{document} 

\title{Weak solutions for the $\al$--Euler equations and convergence to Euler}
\author{Adriana Valentina Busuioc and Dragoş Iftimie}
\date{}

\maketitle

\abstract

We consider the limit $\al\to0$ for the $\alpha$--Euler equations in a two-dimensional bounded domain with Dirichlet boundary conditions. Assuming that the vorticity is bounded in $L^p$, we prove the existence of a global solution and we show  the convergence towards a solution of the incompressible Euler equation with $L^p$ vorticity. The domain can be multiply-connected. We also discuss the case of the second grade fluid when both $\al$ and $\nu$ go to 0.

\section{Introduction}

We consider in this paper the incompressible $\al$--Euler  equations:
\begin{equation}\label{alphaeuler}
  \partial_t (u-\al\Delta u)+u\cdot\nabla(u-\al\Delta u)+\sum_j(u-\al\Delta u)_j\nabla u_j=-\nabla \pi, \qquad \dive u=0,
\end{equation}
on a 2D smooth domain $\Om$ and assuming Dirichlet boundary conditions:
\begin{equation}\label{Dir}
u\bigl|_{\partial\Om}=0.  
\end{equation}

We will not discuss in detail the significance of the $\al$--Euler equations as this was addressed in many other papers. We will simply mention three important facts:
\begin{itemize}
\item The $\al$--Euler equations are the vanishing viscosity case of the second grade fluids found in \cite{dunn_thermodynamics_1974};
\item Like the incompressible Euler equations, the $\al$--Euler equations describe geodesic motion on the group of volume preserving diffeomorphisms for a metric containing the $H^1$ norm of the velocity, see \cite{holm_euler-poincare_1998-1}.
\item The $\al$--Euler equations can also be obtained via an averaging procedure in the Euler equations, see \cite{holm_euler-poincare_1998-1}.
\end{itemize}

When setting $\al=0$ in \eqref{alphaeuler} we formally obtain the incompressible Euler equations
\begin{equation}
  \label{euler}
\dt \ub+\ub\cdot\nabla \ub=-\nabla \overline \pi,\qquad\dive \ub=0.  
\end{equation}
However the Dirichlet boundary conditions \eqref{Dir} are not compatible with the Euler equations. Instead me must use the no-penetration boundary conditions
\begin{equation}
  \label{tan}
\ub\cdot n \bigl|_{\partial\Om}=0
\end{equation}
where $n$ is the unit exterior normal to $\partial\Om$.

A natural question is whether we have convergence of the solutions of \eqref{alphaeuler}--\eqref{Dir} towards solutions of \eqref{euler}--\eqref{tan} when $\al\to0$. The main problem in showing this convergence is the difference in the boundary conditions \eqref{Dir} and \eqref{tan}. Therefore boundary layers are expected to appear. In addition,  we can't have strong estimates uniform in $\al$ for the solutions of \eqref{alphaeuler}. More precisely, the solutions of \eqref{alphaeuler} can't be bounded in any space where the trace to the boundary $\partial\Om$ is well-defined. 

Let us mention at this point that in the case of Navier boundary conditions the boundary layers are weaker and we were able to show in \cite{busuioc_incompressible_2012} the expected convergence, see also \cite{busuioc_uniform_2016} for the case of the dimension three. The case of the Dirichlet boundary conditions was only recently dealt with, and only in dimension two and for a simply-connected domain. More precisely, the authors of \cite{lopes_filho_convergence_2015} were able to adapt the Kato criterion for the vanishing viscosity limit, see \cite{kato_remarks_1984}, to the case of the $\al\to0$ limit obtaining the following result.
\begin{theorem}[\cite{lopes_filho_convergence_2015}]%\label{tl}
Assume that $\Om\subset\R^2$ is simply-connected. Let $\ub_0\in H^3(\Om)$ be divergence free and tangent to the boundary. Assume that $u_0^\al$ verifies
\begin{itemize}
\item $\ualz\in H^3(\Om)$, $\dive u_0^\al=0$ and $u_0^\al\bigl|_{\partial\Om}=0$;
\item $\al^{\frac12}\nlo2{\nabla \ualz}\to0$ and $\al^{\frac32}\|\ualz\|_{H^3(\Om)}$ is bounded as $\al\to0$;
\item $\ualz\to \ub_0$ in $L^2(\Om)$ as $\al\to0$. 
\end{itemize}
Then the unique global $H^3$ solution of \eqref{alphaeuler}--\eqref{Dir} with initial data $\ualz$ converges  in $L^\infty_{loc}([0,\infty);L^2(\Om))$ as $\al\to0$ towards the unique global $H^3$ solution of \eqref{euler}--\eqref{tan} with initial data $\ub_0$.
\end{theorem}

Due to the method of proof, the Kato criterion, it seems that the approach of \cite{lopes_filho_convergence_2015} can only prove convergence towards a $H^3$ solution of the incompressible Euler equation. But other solutions of the Euler equations exist: the Yudovich solutions with bounded vorticity, the weak solutions with $L^p$ vorticity and the vortex sheet solutions where the vorticity is a measure. Our first aim in this paper is to prove that convergence still holds when the limit solution is a weak solution of the Euler equation with $L^p$ vorticity. A secondary aim is to be able to consider multiply-connected domains and also to construct weak solutions of \eqref{alphaeuler}--\eqref{Dir}. Our main result reads as follows.

\begin{theorem}\label{mainthm}
Let $\Om$ be a smooth bounded domain of $\R^2$ and $1<p<\infty$. Let $\ub_0\in W^{1,p}(\Om)$ be divergence free and tangent to the boundary. Let $u_0^\al$ be such that
\begin{itemize}
\item $u_0^\al\in W^{3,p}(\Om)$, $\dive u_0^\al=0$ and $u_0^\al\bigl|_{\partial\Om}=0$;
\item $\al^{\frac12}\nlo2{\nabla u_0^\al}$ and $\nlo p{\curl(u^\al_0-\al\Delta u^\al_0)}$ are bounded independently of $\al$;
\item $u_0^\al\to \ub_0$ in $L^2(\Om)$ as $\al\to0$. 
\end{itemize}
Then there exists a global solution $\ual\in L^\infty(\R_+;W^{3,p}(\Om))$ of \eqref{alphaeuler}--\eqref{Dir}. Moreover, there exists a subsequence of solutions $u^{\al_k}$ and a global solution $\ub$ of the Euler equations \eqref{euler}--\eqref{tan} with initial data $\ub_0$ and vorticity bounded in $L^p(\Om)$, $\curl\ub\in L^\infty(\R_+;L^p(\Om))$, such that $\ualk\to \ub$ in  $L^\infty_{loc}([0,\infty);W^{s,p}(\Om))$ for all $s<\frac1p$ as $\al_k\to0$.
\end{theorem}

A few remarks are needed to understand this result. Let us observe first that the conclusion of this theorem can't be true for $s>\frac1p$. Indeed, suppose that $\ualk\to u$ in  $L^1_{loc}([0,\infty);W^{s,p}(\Om))$ for some $s>\frac1p$. Because functions in $W^{s,p}(\Om)$ with  $s>\frac1p$ have traces on the boundary, we infer that we also have convergence of the traces on the boundary. Since $\ualk$ vanishes on the boundary we therefore deduce that the limit solution $u$ must vanish on the boundary. This is of course not true for solutions of the Euler equations unless some very special situation arises. In general, $\ual$ is unbounded in $L^r_{loc}([0,\infty);W^{s,p}(\Om))$ for any $s>\frac1p$ and $r>1$. 

A second remark is that, by Sobolev embeddings, we have that $\ualk\to u$ in the space $L^\infty_{loc}([0,\infty);L^2(\Om))$ as in the result of \cite{lopes_filho_convergence_2015}. More generally, we obtain that $\ualk\to u$   in $L^\infty_{loc}([0,\infty);H^{s}(\Om))$ for all $s<\min(1-\frac1{p},\frac12)$. 

A third remark is that even though we assume $p<\infty$ in Theorem \ref{mainthm}, it is quite easy to obtain a similar result for the case $p=\infty$. Modifying slightly the conclusion, we can prove in this case convergence towards the Yudovich solution of the Euler equation. Moreover, the Yudovich solutions are unique so we get convergence of the whole sequence $u^\al$ and not only for a subsequence $\ualk$. More details can be found in Remark \ref{yudovich}.
 
Let us comment now on the existence part of Theorem \ref{mainthm}. The main improvement about existence of solutions is that we allow the domain $\Om$ to be multiply-connected and moreover we construct weak solutions. As far as we know, all previous global existence results for $\al$--Euler or second grade fluids with Dirichlet boundary conditions are given for simply-connected domains or with some conditions on the coefficient $\al$ and the initial data, see \cite{cioranescu_existence_1984,cioranescu_weak_1997,galdi_existence_1993,galdi_further_1994}. Here we deal with multiply-connected domains by keeping track of the circulations of $u-\al\Delta u$ on the connected components of the boundary and by exploiting the transport equation that the vorticity $q=\curl(u-\al\Delta u)$ verifies.
Even for simply-connected domains the existence part of Theorem \ref{mainthm} is new, although in the absence of boundaries \cite{oliver_vortex_2001} shows global existence of solutions in the full plane if the initial vorticity $\curl(u_0-\al\Delta u_0)$ is a bounded measure.

Our initial goal was to improve the result of \cite{lopes_filho_convergence_2015}. We achieved that in several aspects. The most important one is that we allow for much weaker solutions of the Euler equations, \textit{i.e.} we prove convergence towards weak solutions with $L^p$ vorticity instead of $H^3$ solutions. The second improvement is that we prove stronger convergence, \textit{i.e.} we prove  strong convergence in $W^{s,p}(\Om)$, $s<\frac1p$, uniformly in time. The third improvement is that we allow for multiply-connected domains. 

Our approach is quite different from that of \cite{lopes_filho_convergence_2015}. In \cite{lopes_filho_convergence_2015}, the authors make a direct estimate of the $L^2$ norm of  $\ual-\ub$ via energy estimates. Here, we use compactness methods and we obtain the required estimates uniform in $\al$ by using the analyticity of the Stokes semi-group.

The plan of the paper is the following. In the next section we introduce some notation, we recall some basic facts about the Stokes operator and prove some preliminary results. In Section \ref{sect3} we prove our main estimates with constants independent of $\al$. These estimates rely on the analyticity of the Stokes semi-group. In Section \ref{sect4} we prove the existence part of Theorem \ref{mainthm}. In Section \ref{sect5} we pass to the limit $\al\to0$ and complete the proof of Theorem \ref{mainthm}. We end this paper with Section \ref{sect6} where we extend Theorem \ref{mainthm} to the case of second grade fluids.

\section{Notation and preliminary results}

Throughout this paper $C$ denotes a generic constant \textit{independent of $\al$} (except in Section \ref{sect4} where it is allowed to depend on $\al$) whose value may change from one relation to another. 

All function spaces are defined on $\Om$ unless otherwise specified.
We denote by $L^p=L^p(\Om)$, $H^1=H^1(\Om)$, $W^{s,p}=W^{s,p}(\Om)$ the usual Sobolev
spaces with the usual norms where for $s$ non-integer $W^{s,p}$ is
defined by interpolation. We shall also use the $H^1_\al$ norm
defined by
\begin{equation*}
\nhua u=\bigl( \nl2u^2+\al\nl2{\nabla u}^2\bigr)^{\frac12}.  
\end{equation*}

The space $L^p_\sigma=L^p_\sigma(\Om)$ is the subspace of $L^p $ formed by
all divergence free and
tangent to the boundary vector fields. We endow $L^p_\sigma $ with
the $L^p$ norm. Recall that the divergence free
condition allows to define the normal trace of a vector field on the
boundary. 

We denote by $\PP$ the standard Leray projector, that is the $L^2$
orthogonal projection from $L^2$ to $L^2_\sigma$. It is well-known that $\PP$
extends by density to a bounded operator from $L^p $ to $L^p_\sigma $.

We denote by $\A=-\PP\Delta$ the classical Stokes operator that we view
as an unbounded operator on $L^p_\sigma$. It is well-known that the domain
of $\A$ is
\begin{equation*}
D(\A)=\{u\in W^{2,p}\ ;\ \dive u=0\text{ and }u\bigl|_{\partial\Om}=0\}.  
\end{equation*}

We know that for any $\la\in\C\setminus(-\infty,0)$ the operator
$\la+\A$ is invertible and for any $f\in L^p_\sigma$ we have that
$(\la+\A)^{-1}\in D(\A)$, see for example \cite[Proposition
2.1]{giga_analyticity_1981}.

A property that will be crucial in what follows is the analyticity of
the Stokes semi-group which can be expressed in terms of the following
inequality proved in \cite[Theorem 1]{giga_analyticity_1981}:
\begin{theorem}[\cite{giga_analyticity_1981}]\label{analyticity}
For any $\ep>0$ there exists a constant $C_\ep$ such that for all
$\la\in\C\setminus\{0\}$, $|\arg\la|\leq\pi-\ep$, and for all $f\in L^p_\sigma$ the
following inequality holds true:
\begin{equation*}
\nl p{(\la+\A)^{-1})f}\leq \frac{C_\ep}{|\la|}\nl pf.
\end{equation*}
\end{theorem}

We will also need to characterize the domains of the powers of $\A$. The following proposition is a consequence of \cite[Theorem
3]{giga_domains_1985} and of the results of \cite{fujiwara_asymptotic_1969}.
\begin{proposition}\label{domain}
Let  $0\leq s\leq 2$. We have that $D(\A^{\frac s2})\hookrightarrow W^{s,p}$. Moreover,
\begin{equation*}
D(\A^{\frac s2})= \{f\in W^{s,p}\ ;\ \dive f=0\text{ and f is tangent to the boundary }\}\quad\text{if}\quad s<\frac1p 
\end{equation*}
and
\begin{equation*}
D(\A^{\frac s2})=\{f\in W^{s,p}\ ;\ \dive f=0\text{ and }f\bigl|_{\partial\Om}=0\}\quad\text{if}\quad s>\frac1p.   
\end{equation*}
\end{proposition}

We assume that  $\Om$ is a smooth domain with holes. The boundary  $\partial\Om$ has a finite number of connected
components which are closed curves. We denote by $\Gamma$ the outer connected component and by
$\Ga_1,\dots,\Ga_N$ the inner connected components. In other words, we have that
$\partial\Om=\Ga\cup\Ga_1\cup\dots\cup\Ga_N$ where $\Ga,\Ga_1,\dots,\Ga_N$
are smooth closed curves and $\Ga_1,\dots,\Ga_N$ are located inside $\Ga$. We denote by $n$ the unit outer normal to $\partial\Om$.

We continue with a remark on the circulations of $v=u-\al\Delta u$ on each connected component of the boundary. The circulation of $v$ on $\Gamma_i$ is defined by $\int_{\Gamma_i}v\cdot n^\perp$.
\begin{lemma}\label{circulation}
Let $u$ be a sufficiently smooth solution of \eqref{alphaeuler}--\eqref{Dir}. Then for every $i\in\{1,\dots,n\}$ the circulation of $v=u-\al\Delta u$ on $\Gamma_i$ is conserved in time.   
\end{lemma}
\begin{proof}
The vector field $v$ verifies the following PDE:
\begin{equation*}
  \dt v+u\cdot\nabla v+\sum_jv_j\nabla u_j=-\nabla \pi.
\end{equation*}
We multiply by $n^\perp$ and integrate on $\Gamma_i$. Recalling that $u$ vanishes on the boundary we get
\begin{equation*}
\dt\int_{\Gamma_i}v\cdot n^\perp+\int_{\Gamma_i}n^\perp \cdot\nabla u\cdot v=-\int_{\Gamma_i}n^\perp\cdot\nabla \pi.
\end{equation*}
Since $n^\perp \cdot\nabla$ is a tangential derivative and $u$ vanishes on the boundary we have that $n^\perp \cdot\nabla u=0$ at the boundary so the second term above vanishes. Finally, using that $n^\perp $ is the unit tangent vector field and recalling that $\Gamma_i$ is a closed curve we infer that the term on the right-hand side vanishes too. This completes the proof.
\end{proof}

We recall now some well-known facts about harmonic vector fields, we  refer to \cite{lopes_filho_vortex_2007,auchmuty_$l^2$_2001,auchmuty_bounds_2016} and the references therein for
details. We call harmonic vector field an $L^p$ vector field which is divergence free, curl free and tangent to the boundary. A harmonic vector field is smooth and the space of all harmonic vector fields is finite dimensional of dimension $N$. A harmonic vector field is uniquely determined by its circulations on $\Gamma_1,\dots,\Ga_N$. A basis of the space of harmonic vector fields is given by $\{Y_1,\dots,Y_N\}$ where $Y_i$ is the unique harmonic vector field with vanishing
circulation on all $\Gamma_1,\dots,\Ga_N$ except on $\Ga_i$ where the circulation
must be 1. 

If $f$ is a divergence free vector field tangent to the boundary we define $\wt f$ to be the unique vector field of the form
\begin{equation}\label{tilde}
\wt f=f-\sum_{i=1}^Na_iY_i  
\end{equation}
where the $a_i$ are constants and $\wt f$ has vanishing circulation on all $\Gamma_1,\dots,\Ga_N$. Equivalently, the constant $a_i$ is the circulation of $f$ on $\Gamma_i$.

We conclude this preliminary section with a Poincaré-like inequality.
\begin{lemma}\label{lemmapoincare}
Suppose that $f\in W^{1,p}$ is a divergence free vector field tangent to the
boundary such that  its circulation on each $\Gamma_1,\dots,\Ga_N$ vanishes. Then there exists some constant $C$ that depends only on $\Om$ and $p$ such that
\begin{equation}\label{poincarest}
\sob1pf\leq C\nl p{\curl f}.  
\end{equation}
\end{lemma}
\begin{proof}
Since $f$ is divergence free and tangent to the boundary, we know from
classical elliptic estimates that the following inequality holds true:
\begin{equation}\label{foias}
\sob1pf\leq C(\nl pf+\nl p{\curl f})  
\end{equation}
so in order to prove \eqref{poincarest} it suffices to show that
\begin{equation}
  \label{poincare}
\nl pf\leq C\nl p{\curl f}.    
\end{equation}
Assume by absurd that \eqref{poincare} fails to be true. Then  \eqref{poincare} fails for $C=n$ so there exists a sequence $f_n$ of divergence free vector fields tangent to the boundary with vanishing circulation on each   $\Gamma_1,\dots,\Ga_N$ and such that
\begin{equation*}
\nl p{f_n}=1\quad\text{and}\quad \nl p{\curl f_n}\leq\frac1n.  
\end{equation*}
Using the estimate \eqref{foias} for $f_n$ we see that $f_n$ is bounded
in $W^{1,p}$. Using the compact embedding of $W^{1,p} $ into
$L^p $ we deduce that there exists a subsequence $f_{n_k}$ and
some $f\in W^{1,p} $ such that $f_{n_k} \to f$ weakly in
$W^{1,p} $ and strongly in $L^p $. In particular $\nl
pf=1$. Moreover, $f_{n_k}$ being divergence free, tangent to the
boundary with vanishing circulation on each  $\Gamma_i$ and the weak
convergence in $W^{1,p}$ imply that so is $f$. Moreover, since $\curl
f_{n_k}\to0$ we have that $f$ is also curl free. So $f$ is a harmonic vector field  with vanishing
circulation on each  $\Gamma_i$. This implies that $f=0$ which is a contradiction because $\nl pf=1$. This completes the proof.
\end{proof}

\section{Main estimate}
\label{sect3}

In this section we consider some vector field $u\in W^{3,p} $ which is divergence free and  vanishing on the boundary $\partial\Om$. We define $v=u-\al\Delta u$. The aim of this section is to estimate $u$ as best as possible in terms of $\nhua u$ and of $\nl p{\curl v}$ with constants independent of $\al$. To do that we will distinguish two parts in $u$: one part that comes from $\curl v$ and another part which comes from the circulations of $v$ on $\Ga_1,\dots,\Ga_N$.

We observe first that $v$ and $\PP v$ have the same circulation on each $\Gamma_i$. Indeed, the Leray decomposition says that $v$ and $\PP v$ differ by a gradient
\begin{equation}\label{leray}
v=\PP v+\nabla \pi'  
\end{equation}
so
\begin{equation*}
\int_{\Gamma_i}n^\perp\cdot v-\int_{\Gamma_i}n^\perp\cdot \PP v=\int_{\Gamma_i}n^\perp\cdot\nabla \pi'=0  
\end{equation*}
where we used that  $n^\perp $ is the unit tangent vector field and $\Gamma_i$ is a closed curve.

Let us define $\widetilde{\PP v}$ as in relation \eqref{tilde}, that is
\begin{equation}\label{pvt}
\wt {\PP v} =\PP v-\sum_{i=1}^N\gamma_iY_i  
\end{equation}
where $\gamma_i$ is the circulation of $v$ on $\Gamma_i$:
\begin{equation*}
\gamma_i= \int_{\Gamma_i}n^\perp\cdot v. 
\end{equation*}

Recall that the operator $1+\al\A$ is invertible. Let us introduce $\wt u=(1+\al\A)^{-1}\wt{\PP v}$ so that
\begin{equation}\label{defutilde1}
\wt u+\al\A\wt u=\wt{\PP v}.  
\end{equation}

Now, let us consider some scalar function $q\in L^p $. By the Biot-Savart law, there exists some divergence-free vector field $f\in W^{1,p}$ tangent to the boundary such that $\curl f=q$. Adding a suitable linear combination of $Y_i$, we can further assume that $f$ has vanishing circulation on each  $\Gamma_1,\dots,\Ga_N$. Moreover, $f$ is the unique vector field verifying all these properties. We denote $f=\S(q)$. This allows us to define the vector field $\T(q)=(1+\al\A)^{-1}f=(1+\al\A)^{-1}\S(q)$.

With the notation introduced above, we remark that if $q=\curl(u-\al\Delta u)$ then
\begin{equation*}
  \T(q)=\tu.
\end{equation*}
Indeed, we have that
\begin{equation*}
q=\curl (u-\al\Delta u)=\curl v=\curl \PP v=\curl \widetilde{\PP v} 
\end{equation*}
where we used that \eqref{leray}, \eqref{pvt} and the fact that the $Y_i$ are curl free. The vector field $\widetilde{\PP v} $ is divergence free, tangent to the boundary, has vanishing circulation on each  $\Gamma_1,\dots,\Ga_N$ and its curl is $q$. So $\widetilde{\PP v}=\S(q)$ and the definition of $\T(q)$ allows to conclude that $\T(q)=(1+\al\A)^{-1}\widetilde{\PP v}=\tu$.

We assume in the sequel that $q=\curl(u-\al\Delta u)$.

Let us apply the Leray projector $\PP$ to the relation $v=u-\al\Delta u$. We get $\PP v=u+\al\A u$ that is $u=(1+\al\A)^{-1}\PP v$. Let us apply $(1+\al\A)^{-1}$ to \eqref{pvt}. We obtain
\begin{equation}\label{XX}
u=(1+\al\A)^{-1}\PP v  =(1+\al\A)^{-1}\wt{\PP v}+  \sum_{i=1}^N(1+\al\A)^{-1}\gamma_iY_i=\T(q)  +  \sum_{i=1}^N\gamma_i\txi
\end{equation}
where we defined
\begin{equation*}
  \txi=(1+\al\A)^{-1}Y_i.
\end{equation*}

We will estimate separately the part due to the vorticity $q$, \textit{i.e.} $\T(q)$, and the part due to the circulations $\gamma_1,\dots,\gamma_N$, \textit{i.e.} $\sum_{i=1}^N\gamma_i\txi$.

We start by estimating $\T(q)$.
\begin{proposition}\label{mainprop}
Let $q\in L^p $. Then $\T(q)\in W^{3,p}$ with $\sob3p{\T(q)}\leq C(\al)\nl pq$. Moreover, for any $\ep>0$ there exists a constant $C$ that
depends only on $\Om$, $p$ and $\ep$ but not on $\al$ such that
\begin{equation}
  \label{eq:88}
\sob{\frac1p-\ep}p{\T(q)}\leq C\nl pq,
\end{equation}
\begin{equation}
  \label{eq:8}
\sob1p{\T(q)}\leq C\al^{-\frac12+\frac1{2p}-\ep}\nl pq,
\end{equation}
\begin{equation}\label{bound}
  \nhu{\T(q)}\leq C\al^{\min(-\frac1{2p},-\frac14)-\ep}\nl pq
\end{equation}
and
\begin{equation}\label{bound0}
\nh{\min(1-\frac1p,\frac12)-\ep}{\T(q)}\leq C \nl pq. 
\end{equation}
\end{proposition}
\begin{proof}
Let $f=\S(q)$, \textit{i.e.} $f$ is the unique vector field which is  divergence-free, tangent to the boundary, with vanishing circulation on each  $\Gamma_1,\dots,\Ga_N$ and such that $\curl f=q$. Lemma \ref{lemmapoincare} implies that
\begin{equation}\label{pvbound}
\sob1p{f}\leq C\nl pq.  
\end{equation}

Recall that $\T(q)=(1+\al\A)^{-1}f$. Since $f\in W^{1,p} $ we deduce from the classical regularity results for the (elliptic) Stokes operator (see \cite{cattabriga_su_1961}) that $\T(q)\in W^{3,p}$. In addition, we have the bound $\sob3p{\T(q)}\leq C(\al)\sob1pf\leq C(\al)\nl pq$.

Let $s\in (0,\frac1p)$ (the value of $s$ will be chosen later). We
deduce from Proposition \ref{domain} that $f\in D(\A^{\frac s2})$ and moreover
\begin{equation*}
\nl p{\A^{\frac s2} f}\leq C  \sob sp{ f}\leq C\sob1p{ f}\leq C\nl pq.  
\end{equation*}

Recall that  $\T(q)=(1+\al\A)^{-1} f$ so 
\begin{equation}\label{defutilde}
\T(q)+\al\A \T(q)= f.  
\end{equation}
Clearly $ \T(q)\in D(\A^{\frac s2})$ and since  $ f$ also belongs to  $D(\A^{\frac s2})$  we infer that $\A \T(q)\in D(\A^{\frac s2})$. We can therefore apply the operator $\A^{\frac s2}$ to relation \eqref{defutilde} to obtain that
\begin{equation*}
\A^{\frac s2} \T(q)+\al\A\A^{\frac s2} \T(q)=\A^{\frac s2} f.    
\end{equation*}
We used above that the operators $\A$ and $\A^{\frac s2}$ commute. 

We will use now the analyticity of the Stokes semi-group stated in
Theorem \ref{analyticity} to deduce that
\begin{equation}\label{ffff}
\nl p{\A^{\frac s2} \T(q)}+\al \nl p{\A^{1+\frac s2} \T(q)} \leq \nl p{\A^{\frac s2} f} \leq C\nl pq.  
\end{equation}

According to Proposition \ref{domain} we have the embedding $D(\A^{\frac s2})\hookrightarrow W^{s,p} $  so we can further deduce that
\begin{equation*}%\label{6}
\sob sp{ \T(q)}\leq C \nl p{\A^{\frac s2} \T(q)}\leq C\nl pq. 
\end{equation*}
This proves relation \eqref{eq:88}.

Next, let us observe that we have the following estimate: 
$$\sob {2+s}p{ \T(q)}\leq C\nl p{\A^{1+\frac s2} \T(q)}.$$ 
Indeed, if $\T(q)\in D(\A^{1+\frac s2})$ then $\A\T(q)\in D(\A^{\frac s2})$ so Proposition \ref{domain} implies again that $\A\T(q)\in W^{s,p}$. The classical regularity results for the Stokes operator (see \cite{cattabriga_su_1961}) imply then that $\T(q)\in W^{2+s,p}$ with the required inequalities. Therefore relation \eqref{ffff} also yields
\begin{equation*}
\sob {2+s}p{ \T(q)}\leq  \frac{C}{\al}\nl pq. 
\end{equation*}
This proves relation \eqref{eq:88}. Next, we infer by interpolation that 
\begin{equation}\label{10}
\sob {s'}p{ \T(q)}\leq C \sob sp{ \T(q)}^{1-\frac{s'-s}2}\sob {2+s}p{ \T(q)}^{\frac{s'-s}2}  \leq  C\al^{\frac{s-s'}2} \nl pq 
\end{equation}
for all $s'\in[s,2+s]$. 

We first choose $s'=1$ and $s=\frac1p-2\ep$ in \eqref{10} and we get \eqref{eq:8}. 

We prove now the bound \eqref{bound0}. If $p\leq 2$ this relation follows from \eqref{eq:88}  and from the Sobolev embedding $W^{\frac1p-\ep,p}\hra H^{1-\frac1p-\ep}$ (for $\ep$ sufficiently small). If $p\geq2$ we have that $q\in L^p\hra L^2$ and \eqref{bound0} follows from  the relation \eqref{eq:88} written for $p=2$.

Finally, let us prove \eqref{bound}. Assume first  $p\leq 2$. Choose $s'=2/p$ in \eqref{10} (which is
possible because $s<1/p$). Recalling the Sobolev embedding
$W^{\frac2p,p} \hookrightarrow H^1 $ we deduce that
\begin{equation*}
\nh1{ \T(q)}\leq C\sob {\frac2p} p { \T(q)} \leq  C\al^{\frac s2 -\frac1p} \nl pq.
\end{equation*}
Choosing $s=\frac1p-2\ep$ implies \eqref{bound}. The case $p>2$ follows from the case $p=2$ since if $q$ belongs to $L^p$ then it also belongs to $L^2$ so one can use \eqref{eq:8} for $p=2$. This completes the proof.
\end{proof}

We continue with the estimate of the part of $u$ due to the circulations $\gamma_1,\dots,\gamma_N$.
\begin{proposition}\label{propcirc}
Let $u\in W^{3,p}$ be divergence free and tangent to the boundary. There exists a constant $C$ that depends only on $\Om$ and $p$ such that
\begin{equation*}
\sum_{i=1}^N|\gamma_i|\leq C(\nhua u+\nl pq).  
\end{equation*}
\end{proposition}
\begin{proof}
Let us denote
\begin{equation*}
g=  u-\wt u= \sum_{i=1}^N\gamma_i\txi
\end{equation*}
and
\begin{equation*}
  X=\sum_{i=1}^N\gamma_i Y_i
\end{equation*}
so that 
\begin{equation}\label{eqg}
  g+\al\A g=X.
\end{equation}

Now we estimate the  $H^1_\al$ norm of $g$. This  requires to estimate the
$H^1_\al$ norm of $\wt u$. To do that, let us multiply
\eqref{defutilde1} by $\wt u$ and integrate. We obtain
\begin{equation}\label{h1}
\nl2{\wt u}^2+\al\int_\Om \A\wt u\cdot\wt u=\int_\Om\wt{\PP v}\cdot\wt u.  
\end{equation}
We have that
\begin{equation*}
\int_\Om \A\wt u\cdot\wt u
= \int_\Om \A^{\frac12}\wt u\cdot\A^{\frac12}\wt u
=\nl2{\A^{\frac12}\wt u}^2
= \nl2{\nabla \wt u}^2
\end{equation*}
by the usual properties of the Stokes operator. We infer from
\eqref{h1} that
\begin{equation*}
\nhua{\wt u}^2 =\int_\Om\wt{\PP v}\cdot\wt u\leq \nl2 {\wt{\PP v}}
\nl2{\wt u}
\leq \nl2 {\wt{\PP v}} \nhua{\wt u}
\end{equation*}
so that
\begin{equation*}
  \nhua{\wt u}\leq \nl2 {\wt{\PP v}}.
\end{equation*}
The Sobolev embedding $W^{1,p}(\Om)\hookrightarrow L^2(\Om)$ together with the
bound given in \eqref{pvbound} (recall that $f=\widetilde{\PP v}$) imply
\begin{equation*}
  \nhua{\wt u}\leq \nl2 {\wt{\PP v}}\leq C\sob1p{\wt{\PP v}}\leq C\nl pq.  
\end{equation*}

In the end we get
\begin{equation*}
\nhua{g}=\nhua{u-\wt u}\leq \nhua u+\nhua{\wt u}\leq   \nhua u+ C\nl pq.
\end{equation*}
This implies
\begin{equation}\label{g1}
\nl2g\leq  \nhua u+ C\nl pq 
\end{equation}
and 
\begin{equation}\label{g2}
\nl2{\A^{\frac12}g}=\nl2{\nabla g}\leq  \al^{-\frac12}\nhua{g}\leq  \al^{-\frac12}(\nhua u+ C\nl pq) .
\end{equation}

We apply now $\A^{-\frac12}$ to \eqref{eqg}  and we take the $L^2$
norm to obtain
\begin{equation}\label{g4}
\nl2{\A^{-\frac12}X}=\nl2{\A^{-\frac12}g+\al\A^{\frac12}g}  
\leq \nl2{\A^{-\frac12}g}+\al\nl2{\A^{\frac12}g}.  
\end{equation}

Since $g=(1+\al\A)^{-1}X$ we have that $g\in D(\A)$. So $\A^{-\frac12}g\in D(\A^{\frac12})$ and from Proposition \ref{domain} we deduce that $\A^{-\frac12}g$ vanishes on the boundary. Therefore we can apply the Poincaré inequality to deduce that
\begin{equation}\label{g3}
\nl2{\A^{-\frac12}g}\leq C\nl2{\nabla \A^{-\frac12}g}=C\nl2{\A^{\frac12}\A^{-\frac12}g}=C\nl2g.  
\end{equation}

Using relations \eqref{g1}, \eqref{g2} and \eqref{g3} in \eqref{g4} yields
\begin{equation}\label{X}
\nl2{\A^{-\frac12}X}\leq  C( \nhua u+ \nl pq) + \al^{\frac12}(\nhua u+
C\nl pq) \leq  C( \nhua u+ \nl pq) .
\end{equation}

Because the vector fields $Y_1,\dots,Y_N$ are linearly independent,
one can easily check that the application
\begin{equation*}
  \R^N\ni(a_1,\dots,a_N)\mapsto \nl2{\A^{-\frac12}(\sum_{i=1}^N a_i Y_i)}
\end{equation*}
is a norm on $\R^N$. Because all norms on $\R^N$ are equivalent, there
exists some constant $C$ such that
\begin{equation*}
\sum_{i=1}^N |\gamma_i|\leq  C\nl2{\A^{-\frac12}(\sum_{i=1}^N \gamma_i Y_i)}  
=C\nl2{\A^{-\frac12}X}\leq  C( \nhua u+ \nl pq)
\end{equation*}
where we used \eqref{X}. This completes the proof.
\end{proof}

Putting together Propositions \ref{mainprop} and \ref{propcirc} allows to estimate the full velocity.
\begin{proposition}\label{propvelo}
Let $u\in W^{3,p} $ be divergence free and vanishing on the boundary and let $\ep>0$. There exists a constant $C$ that
depends only on $\Om$, $p$ and $\ep$ but not on $\al$ such that
\begin{equation}
  \label{eq:88u}
\sob{\frac1p-\ep}p{u}\leq C(\nl pq+\nhua u),
\end{equation}
\begin{equation}
  \label{eq:8u}
\sob1p{u}\leq C\al^{-\frac12+\frac1{2p}-\ep}(\nl pq+\nhua u),
\end{equation}
\begin{equation}\label{boundu}
  \nhu{u}\leq C\al^{\min(-\frac1{2p},-\frac14)-\ep}(\nl pq+\nhua u)
\end{equation}
and
\begin{equation}\label{bound0u}
\nh{\min(1-\frac1p,\frac12)-\ep}{u}\leq C (\nl pq+\nhua u). 
\end{equation}  
\end{proposition}
\begin{remark}
The important thing to observe is that the power of $\al$ in
\eqref{boundu} can be made strictly larger than $-\frac12$ (even if it is only slightly larger than $-\frac12$ when $p$ is close to 1). The
significance of this will be obvious later in Section~\ref{sect5}. Indeed, we will need to
show that the terms of the form $\al \nabla u\nabla u$ converge to 0 as
$\al\to0$ when $\nhua u$ and $\nl pq$ are bounded. The trivial bound $\nh1u\leq\al^{-\frac12}\nhua
u$only shows that $\al \nabla u\nabla u$ is bounded while
\eqref{boundu} with a power of $\al$  strictly larger than $-\frac12$
 implies that $\al \nabla u\nabla u$ goes to 0.
\end{remark}
\begin{proof}
We already know from Proposition \ref{mainprop} that relations \eqref{eq:88u}, \eqref{eq:8u}, \eqref{boundu} and \eqref{bound0u} hold true with $u$ replaced by $\tu$ on the left-hand side. It remains to show these relations with $u$ replaced by $\sum_{i=1}^N \gamma_i\txi$ on the  left-hand side. Thanks to Proposition \ref{propcirc} it suffices to show that
\begin{gather*}
\sob{\frac1p-\ep}p{\txi}\leq C,\quad
\sob1p{\txi}\leq C\al^{-\frac12+\frac1{2p}-\ep},\quad
\nhu{\txi}\leq C\al^{\min(-\frac1{2p},-\frac14)-\ep}\\
\intertext{and}
\nh{\min(1-\frac1p,\frac12)-\ep}{\txi}\leq C.
\end{gather*}
Since the vector fields $Y_i$ are smooth, tangent to the boundary and independent of $\al$, these bounds can be proved exactly as in the proof of Proposition \ref{mainprop} by reasoning on $Y_i$ instead of $f$.
\end{proof}

\section{Construction of the solution for fixed $\alpha$}
\label{sect4}

In this section the parameter $\al$ is fixed and the constants are allowed to depend on $\al$.

Our aim in this part is to prove the existence part of Theorem \ref{mainthm}. More precisely, we will show the following result.
\begin{theorem}\label{existence}
Suppose that $u_0\in W^{3,p}$ is divergence free and vanishing on the boundary. There exists a unique global $W^{3,p}$ solution  $u\in C^{0,w}_{b}(\R_+;W^{3,p})$ of \eqref{alphaeuler}--\eqref{Dir}. 
\end{theorem}
Above $C^{0,w}_{b}$ stands for weakly continuous bounded functions.

We now proceed with the proof of Theorem \ref{existence}.

\medskip

The uniqueness part of this theorem is quite easy once we observe that, by the Sobolev embedding $W^{3,p}\hookrightarrow W^{1,\infty}$, the solution is Lipschitz. One can subtract the PDEs for two solutions and multiply by the difference of the solutions to observe that one can estimate the $H^1_\al$ norm of the difference and conclude by the Gronwall inequality. The argument is quite standard so we leave the details to the reader.

\medskip

To prove the existence of the solution, we will first find an equivalent formulation of the equations. Let $\gamma_i$ be the circulation of $v_0=u_0-\al\Delta u_0$ on $\Gamma_i$. We know from Lemma \ref{circulation} that $\gamma_i$ is conserved in time. 

Taking the curl of \eqref{alphaeuler} implies that the vorticity $q=\curl(u-\al\Delta u)$ verifies the following transport equation:
\begin{equation}\label{eqq}
  \dt q+u\cdot\nabla q=0.
\end{equation}

Conversely, if \eqref{eqq} holds true and the circulations $\gamma_i$ are conserved then \eqref{alphaeuler} holds true. Indeed, let $F$ denote the left-hand side of \eqref{alphaeuler}. The fact that \eqref{eqq} holds true means that $\curl F=0$. Going back to the proof of Lemma \ref{circulation} we observe that  the circulations $\gamma_i$ being conserved means that the circulations of $F$ on  each $\Gamma_1,\dots,\Ga_N$ vanish. From the properties of the Leray projector we know that $\PP F$ and $F$ differ by a gradient: $\PP F-F=\nabla\pi$. Taking the curl implies that  $\PP f$ is curl free. But $\PP F$ is also divergence free and tangent to the boundary so it must be a harmonic vector field. Since the circulations of  $F$ on $\Gamma_i$ vanish and recalling that a gradient has vanishing circulation on $\Gamma_i$ we deduce that  $\PP v$ has vanishing circulation on each $\Ga_1,\dots,\Ga_N$. Since it is a harmonic vector field it must therefore vanish. We conclude that $F=-\nabla\pi$ and \eqref{alphaeuler} holds true.

Recalling \eqref{XX} we infer that \eqref{alphaeuler}--\eqref{Dir} is equivalent to the following PDE in the unknown $q$:
\begin{equation}
  \label{eqqm}
\dt q+u\cdot\nabla q=0\quad\text{with}\quad u=\T(q)+\sum_{i=1}^N\gamma_i\txi.  
\end{equation}

To complete the proof of Theorem \ref{existence} it suffices to construct a solution $q\in L^\infty(\R_+;L^p)$ of \eqref{eqqm}. Indeed, by the regularity results for the Stokes operator and recalling that $Y_i$ is smooth we deduce that $\txi$ is smooth too. From Proposition \ref{mainprop} we deduce that $\sob3p\tu\leq C\nl pq$. Therefore $u\in L^\infty(\R_+;W^{3,p})$. From the PDE verified by $q$ one can immediately see that $\dt q$ is bounded in the sense of distributions, so $q$ must be continuous in time with values in $\mathscr{D}'$. Since $q\in L^\infty(\R_+;L^p)$ we infer by density of $\mathscr{C}^\infty_0 $ in $L^p$ that $q\in C^{0,w}_{b}(\R_+;L^p)$. Then $u$ is also weakly continuous in time: $u\in C^{0,w}_{b}(\R_+;W^{3,p}).$

\medskip

To solve \eqref{eqqm} we will regularize it by introducing an artificial viscosity. More precisely, for $\ep>0$ let us consider the following PDE
\begin{equation}
  \label{eqqe1}
\dt q^\ep+u^\ep\cdot\nabla q^\ep-\ep\Delta q^\ep=0  \quad\text{with}\quad u^\ep=\T(q^\ep)+\sum_{i=1}^N\gamma_i\txi
\end{equation}
with Dirichlet boundary conditions
\begin{equation}
  \label{eqqe2}
q^\ep\bigl|_{\partial\Om}=0  
\end{equation}
and some smooth initial data
\begin{equation}
  \label{eqqe3}
q^\ep(0)\in\mathscr{C}^\infty_0   
\end{equation}
such that $q^\ep(0)\to q_0$ in $L^p$ as $\ep\to0$. The global existence of smooth solutions of \eqref{eqqe1}--\eqref{eqqe3} can be proved with classical methods, see for instance \cite[Chapter 15]{taylor_partial_1997}. Moreover, the $L^p$ norms of the solutions decrease in time:
\begin{equation*}
\nl p{q^\ep(t)}\leq \nl p{q^\ep(0)}  \quad\forall t\geq0.
\end{equation*}

Using also Proposition \ref{mainprop} we infer that $q^\ep$ is bounded in $L^\infty(\R_+;L^p)$ independently of $\ep$ and $u^\ep$ is bounded in $L^\infty(\R_+;W^{3,p})$ independently of $\ep$. Then we can extract a subsequence of $q^\ep$ that we denote again by $q^\ep$ and some $\qbb\in L^\infty(\R_+;L^p)$ and $\ubb\in L^\infty(\R_+;W^{3,p})$ such that
\begin{equation*}
q^\ep\rightharpoonup\qbb\quad\text{weakly in}\quad L^\infty(\R_+;L^p)  
\end{equation*}
and 
\begin{equation}\label{9}
u^\ep\rightharpoonup\ubb\quad\text{weakly in}\quad L^\infty(\R_+;W^{3,p}).  
\end{equation}

We now pass to the limit in \eqref{eqqe1} in the sense of distributions when $\ep\to0$. Obviously $\dt q^\ep\to\dt q$ and $\ep\Delta q^\ep\to0$ in the sense of distributions when $\ep\to0$. It remains to show that $u^\ep q^\ep\to \ubb\,\qbb$ in the sense of distributions. To do that, we observe first from \eqref{eqqe1} that $\dt q^\ep$ is bounded in $L^\infty(\R_+;W^{-2,p})$. Since the embedding $W^{-2,p}\hookrightarrow W^{-3,p}$ is compact we deduce from the Arzelà-Ascoli theorem that there exists a subsequence of $q^\ep$, again denoted by $q^\ep$, such that $q^\ep\to\qbb$ strongly in  $L^\infty_{loc}([0,\infty);W^{-3,p})$. This strong convergence combined with \eqref{9} implies that $u^\ep q^\ep\to\ubb\,\qbb$ in the sense of distributions. Indeed, the product $(u,q)\mapsto uq$ is continuous from $W^{3,p}\times W^{-3,p}$ into $\mathscr{D}'$ as can be seen from the following estimate:
\begin{equation*}
  \forall\varphi\in C^\infty_0\quad\Bigl|\int_\Om uq\varphi\Bigr|
\leq \|q\|_{W^{-3,p}}\|u\varphi\|_{W^{3,p}}
\leq C\|q\|_{W^{-3,p}}\|u\|_{W^{3,p}}\|\varphi\|_{W^{3,\infty}}.
\end{equation*}

We conclude that we can pass to the limit $\ep\to0$ in \eqref{eqqe1} to deduce that 
\begin{equation*}
  \dt\qbb+\ubb\cdot\nabla\qbb=0.
\end{equation*}
From the uniform in time convergence:  $q^\ep\to\qbb$ strongly in  $L^\infty_{loc}([0,\infty);W^{-3,p})$ we infer that
\begin{equation*}
\qbb(0)=\lim_{\ep\to0}q^\ep(0)=q_0.  
\end{equation*}

To complete the proof of Theorem \ref{existence} it remains to prove that $\ubb=\T(\qbb)+\sum_{i=1}^N\gamma_i\txi$. We know that $q^\ep=\curl(u^\ep-\al\Delta u^\ep)$ so, after passing to the limit $\ep\to0$ in the sense of distributions, we get that $\qbb=\curl(\ubb-\al\Delta \ubb)$. On the other hand, from \eqref{9} we have that $v^\ep\to\vbb=\ubb-\al\Delta\ubb$ weakly in $L^\infty(\R_+;W^{1,p})$. In particular we have convergence of the trace of $v^\ep$ on the boundary to the trace of $\vbb$ on the boundary. So the circulations of $v^\ep$ on each $\Gamma_i$ converge towards the circulations of $\vbb$ on each $\Gamma_i$. We infer that the circulation of $\vbb$ on $\Gamma_i$ is $\gamma_i$. This information combined with the relation $\qbb=\curl(\ubb-\al\Delta \ubb)$ implies that  $\ubb=\T(\qbb)+\sum_{i=1}^N\gamma_i\txi$. This completes the proof of Theorem \ref{existence}.

\section{Passing to the limit $\al\to0$}
\label{sect5}

In this section we show the convergence part of Theorem \ref{mainthm}. Let $\ual$ the solution constructed in Theorem \ref{existence} and let us also denote $v^\al=\ual-\al\Delta \ual$ and $q^\al=\curl v^\al$.

Multiplying \eqref{alphaeuler} by $\ual$ and integrating in space and time
implies that the $H^1_\al$ norm of the  velocity is conserved:
\begin{equation*}
\nhua{\ual(t)}=\nhua{\ualz} \quad\forall t\geq0. 
\end{equation*}
By hypothesis, we know that $\nhua{\ualz}$ is bounded uniformly in $\al$ hence
\begin{equation}\label{l1}
  \ual\text{ bounded in } L^\infty(\R_+;H^1_\al)
\end{equation}
uniformly in $\al$.

Moreover, from the transport equation verified by $q^\al$ we know that the $L^p$ norm of $q^\al$ is also conserved:
\begin{equation*}
\nl p{q^\al(t)}= \nl p{q^\al_0}  \quad\forall t\geq0
\end{equation*}
so
\begin{equation}\label{l2}
  q^\al\text{ bounded in } L^\infty(\R_+;L^p)
\end{equation}
uniformly in $\al$. 

Relation \eqref{l1} implies that $\ual$ is bounded in $L^\infty(\R_+;L^2)$. Using also relation \eqref{l2} we deduce that there exists a subsequence $\ualk$ of $\ual$, some vector field $\ub$ and  some scalar function $\omb$  such that
\begin{equation*}
%  \label{eq:1}
\ualk\rha \ub\quad\text{weakly in }L^\infty(\R_+;L^2) 
\end{equation*}
and
\begin{equation}
  \label{eq:2}
q^{\al_k}\rha\omb\quad\text{weakly in }L^\infty(\R_+;L^p)  . 
\end{equation}

Clearly $\ub$ is divergence free and tangent to the boundary. Since $\al_k\curl\Delta \ualk\to0$ in the sense of distributions we have that $q^{\al_k}=\curl\ualk-\al_k\curl\Delta \ualk\to\curl\ub$ in the sense of distributions. By uniqueness of limits in the sense of distributions we infer that $\omb=\curl\ub$. 

We need to prove that $\ub$ verifies the Euler equation \eqref{euler}. In order to do that, we shall pass to the limit $\al\to0$ in \eqref{alphaeuler}. A simple calculation shows that the $\al$--Euler equations can be written under the following form
\begin{multline}
   \label{4}
\dt (\ual-\al\Delta \ual)+\dive(\ual\otimes \ual)
-\al\sum_{j,i}\partial_j\partial_i(\uall_j\partial_i \ual)
+\al\sum_{j,i}\partial_j(\partial_i \uall_j\partial_i \ual)\\
-\al\sum_{j,i}\partial_i(\partial_i \uall_j\nabla \uall_j)
=-\nabla \pi^\al
\end{multline}
for some $\pi^\al$ (see \cite{iftimie_remarques_2002-1}).

Because $\ualk\to\ub$ in the sense of
distributions we have that $\dt\ualk\to\dt\ub$ in the sense of
distributions and also $\dt\Delta\ualk\to\dt\Delta\ub$ in the sense of
distributions so $\al_k\dt\Delta\ualk\to0$ in the sense of
distributions. Recall also that the limit of a gradient is gradient. 

Let us now show that the last three terms on the left-hand side of \eqref{4} go to 0 in the
sense of distributions. Let us consider for example the term
${\al_k}\partial_j(\partial_i u^{\al_k}_j\partial_i u^{\al_k})$. Thanks to
Proposition \ref{propvelo}  we know that there exists some $\eta<\frac12$ such that $\nh1\ual\leq C\al^{-\eta}(\nhua\ual+\nl p{q^{\al}})$. We bound
\begin{align*}
\nl1{{\al_k}\partial_i u^{\al_k}_j\partial_i u^{\al_k}}
&\leq C{\al_k}\nh1\ualk^2  \\
&\leq C{\al_k}^{1-2\eta}(\nhua\ualk+\nl p{q^{\al_k}})^2\\
&\leq C{\al_k}^{1-2\eta}\\
&\longrightarrow 0
\end{align*}
when $\al_k\to0$. We infer that ${\al_k}\partial_i
u^{\al_k}_j\partial_i u^{\al_k}\to0$ in the sense of distributions so
${\al_k}\sum_{j,i}\partial_j(\partial_i u^{\al_k}_j\partial_i
u^{\al_k})\to0$ in the sense of the distributions. One can show in a
similar manner that the remaining two terms from \eqref{4} with coefficient $\al$ also go to 0 in
the sense of distributions.

It remains to pass to the limit in the term $\ualk\otimes\ualk$. To do that we require compactness of the sequence $\ualk$. This will be obtained via time-derivative estimates. To get these time-derivative estimates it is more practical to work in $L^2$ based function spaces. 

Let $\A_2$ be the Stokes operator seen as an unbounded operator in $L^2$. For $s\geq0$ we define $X^s$ to be domain of $\A_2^{\frac s2}$ with norm $\|f\|_{X^s}=\nl2{\A_2^{\frac s2}f}$. We also define $X^{-s}$ to be the dual space of $X^s$.

Estimates on the time derivative of $\ual-\al\Delta \ual$ are easy to obtain directly from the PDE \eqref{alphaeuler} but we need estimates on $\partial_t\ual$ and we  must be careful about the dependence on $\al$. 

Let us consider a test vector field $\varphi\in X^4$ and let us define $\varphi^\al=(1+\al\AA)^{-1}\varphi$. One can use the classical results about the domain of $\A_2^s$ (see for example \cite[Chapter 4]{constantin_navier-stokes_1988}) to observe that $\varphi^\al\in D(\A_2^3)$. Expressing both $\varphi$ and $\varphi^\al$ in terms of an orthonormal base of eigenfunctions of $\A_2$ as in  \cite[Chapter 4]{constantin_navier-stokes_1988} and using the regularity results in that reference, one can easily see that  we have
\begin{equation}\label{2}
\nh 4{\varphi^\al}\leq C\nl2{ \AA^2  \varphi^\al}\leq C\nl2 {\AA^2\varphi}=C\|\varphi\|_{X^4}.
\end{equation}

Recall that since $\varphi^\al$ is divergence free and tangent to the boundary (even vanishing on the boundary) we have that $\PP\varphi^\al=\varphi^\al$. We multiply \eqref{4} by $\varphi^\al=\PP\varphi^\al$ to obtain
\begin{multline*}
\langle \partial_t(\ual-\al\Delta \ual),\PP\varphi^\al\rangle
=\int_\Om \ual\cdot\nabla\varphi^\al\cdot \ual
+\al \sum_{j,i}\int_\Om \uall_j\partial_i \ual\cdot \partial_j\partial_i\varphi^\al\\
+\al\sum_{j,i}\int_\Om \partial_i \uall_j\partial_i \ual\cdot\partial_j\varphi^\al
-\al\sum_{j,i}\int_\Om \partial_i \uall_j\nabla \uall_j\cdot \partial_i\varphi^\al.
\end{multline*}

We now bound each of these terms. First, by the Hölder inequality and  by
Sobolev embeddings we have that
\begin{equation*}
\bigl|\int_\Om \ual\cdot\nabla\varphi^\al\cdot \ual \bigr|
\leq C\nl{2}{\ual}^2\nl{\infty}{\nabla\varphi^\al}
\leq C\nl{2}{\ual}^2\nh3{\varphi^\al}
\leq C\|\varphi\|_{X^4}
\end{equation*}
where we also used \eqref{l1} and  \eqref{2}.

Similarly,
\begin{align*}
\bigl|\al \sum_{j,i}\int_\Om \uall_j\partial_i \ual\cdot \partial_j\partial_i\varphi^\al\bigr|
&\leq C\al\nl2{\ual}\nh1{\ual}\|\varphi^\al\|_{W^{2,\infty}}\\
&\leq C\al^{\frac12}\nhua{\ual}^2\nh4{\varphi^\al}\\
&\leq C\|\varphi\|_{X^4}
\end{align*}
and
\begin{align*}
  \bigl|\al\sum_{j,i}\int_\Om \partial_i \uall_j\partial_i \ual\cdot\partial_j\varphi^\al
-\al\sum_{j,i}\int_\Om \partial_i \uall_j\nabla \uall_j\cdot \partial_i\varphi^\al\bigr|
&\leq C\al\nh1{\ual}^2\nl\infty{\nabla\varphi^\al}\\
&\leq C\nhua{\ual}^2\nh3{\varphi^\al}\\
&\leq C\|\varphi\|_{X^4}.
\end{align*}

On the other hand, we have that
\begin{align*}
\langle \partial_t(\ual-\al\Delta \ual),\PP\varphi^\al\rangle  
&=\langle \PP\partial_t(\ual-\al\Delta \ual),\varphi^\al\rangle\\  
&=\langle \partial_t(\ual+\al\A_2 \ual),\varphi^\al\rangle  \\
&=\langle \partial_t\ual,(1+\al\A_2)\varphi^\al\rangle  \\
&=\langle \partial_t\ual,\varphi\rangle.  
\end{align*}

We deduce from the above estimates the following bound:
\begin{equation*}
|\langle \partial_t\ual,\varphi\rangle|\leq  C\|\varphi\|_{X^4}.
\end{equation*}

This implies that $\partial_t\ual$ is bounded in $L^\infty(\R_+;X^{-4})$. In particular, the $\ual$ are equicontinuous in time with values in $X^{-4}$. The $\ual$ are also bounded in $X^{-4}$ because by \eqref{l1} they are bounded in $L^2$ and $L^2_\sigma=X^0\hookrightarrow X^{-4}$. Moreover, by compact Sobolev embeddings we know that the embedding $X^{-4}\hookrightarrow X^{-5}$ is compact. Finally, the Arzelà-Ascoli theorem implies that there exists a subsequence of $\ualk$, again denoted by $\ualk$, such that $\ualk(t)\to\ub(t)$ in $X^{-5}$ uniformly in time:
\begin{equation}\label{f11}
\ualk\to\ub\quad\text{strongly in }L^\infty_{loc}([0,\infty);X^{-5}).  
\end{equation}

Thanks to Proposition \ref{propvelo} we know that there exists some $s_0\in(0,\frac12)$ such
that $\ualk$ is bounded in
$L^\infty(\R_+;H^{s_0})$. Therefore in $L^\infty(\R_+;X^{s_0})$
too. Consequently $\ub\in L^\infty(\R_+;X^{s_0})$. By interpolation and using \eqref{f11} we deduce that 
\begin{equation}\label{l2conv}
\ualk\to\ub\quad\text{strongly in }L^\infty_{loc}([0,\infty);L^2).  
\end{equation}

We infer that $\ualk\otimes\ualk\to\ub\otimes\ub$ in the sense of
distributions and therefore
$\dive(\ualk\otimes\ualk)\to\dive(\ub\otimes\ub)$ in the sense of
distributions too.   

We proved that $\ub$ verifies the incompressible Euler equation. Moreover, we recall \eqref{eq:2} which says in particular that $\omb=\curl\ub\in L^\infty(\R_+;L^p)$. To complete the proof of Theorem \ref{mainthm} it suffices to show that $\ualk\to\ub$ in  $L^\infty_{loc}([0,\infty);W^{s,p})$ for all $s<\frac1p$. We consider two cases, depending on $p$ being larger or smaller than 2. 

If $p\leq 2$ we have that $L^2\subset L^p$ so from \eqref{l2conv} we deduce that $\ualk\to\ub$ in  $L^\infty_{loc}([0,\infty);L^p)$. But we know from Proposition \ref{propvelo} and from the boundedness of $\nhua {u^\al}$ and of $\nl p{q^\al}$ that $\ualk$ is bounded in $L^\infty_{loc}([0,\infty);W^{s,p})$ for all $s<\frac1p$. By interpolation we conclude that $\ualk\to\ub$ in  $L^\infty_{loc}([0,\infty);W^{s,p})$ for all $s<\frac1p$.

If $p\geq2$ then $p'\equiv\frac{p}{p-1}\leq2$. So we have the Sobolev embedding $W^{1,p'}_0\hra L^2$. Passing to the dual we obtain that $L^2\hra W^{-1,p}$. Then we deduce from \eqref{l2conv} that $\ualk\to\ub$ in  $L^\infty_{loc}([0,\infty);W^{-1,p})$. We conclude as above by interpolation that $\ualk\to\ub$ in  $L^\infty_{loc}([0,\infty);W^{s,p})$ for all $s<\frac1p$. This completes the proof of Theorem \ref{mainthm}.

\begin{remark}\label{yudovich}
If $p=\infty$ we have that  $L^\infty\subset L^{r}$ for any $r$ so if $q^\al_0$ is bounded in $L^\infty$ it is also bounded in any $L^{r}$ with $r$ finite. Therefore one can still pass to the limit $\al\to0$ using the case $p<\infty$. The limit solution $\ub$ is a Yudovich solution. Indeed, on one hand we know from \eqref{eq:2} that $q^\al$ converges to $\omb$ and on the other hand $q^\al$ is bounded in $L^\infty(\R_+\times\Om)$. So necessarily $\omb\in L^\infty(\R_+\times\Om)$ which implies that $\ub$ is a Yudovich solution. We conclude that Theorem \ref{mainthm} remains true in the case $p=\infty$ with the following modifications in the conclusion:
\begin{itemize}
\item  the solution $u^\al$ belongs to the space $L^\infty(\R_+;W^{3,r} )$ for all $r<\infty$ instead of $L^\infty(\R_+;W^{3,\infty} )$;
\item the convergence holds true in $L^\infty_{loc}([0,\infty);W^{s,r} )$ for all $s<\frac1r$ and $r<\infty$. 
\end{itemize}
\end{remark}

\section{The case of second grade fluids}
\label{sect6}

The equation of motion of second grade fluids read as follows:
\begin{equation}\label{secondgrade}
  \partial_t (u-\al\Delta u)-\nu\Delta u+u\cdot\nabla(u-\al\Delta u)+\sum_j(u-\al\Delta u)_j\nabla u_j=-\nabla \pi, \qquad \dive u=0.
\end{equation}
We endow this equation with the Dirichlet boundary conditions too. We observe that the $\al$--Euler equations are a particular case of second grade fluids, more precisely they are the vanishing viscosity case $\nu=0$. We refer to the recent book \cite{cioranescu_mechanics_2016} for an extensive discussion of various aspects of the second grade fluids.

As for the $\al$--Euler equations, we use the notation $v=u-\al\Delta u$ and $q=\curl v$.

Let us mention at this point that convergence towards a solution of the Euler equation when $\al,\nu\to0$ was proved in the case of the Navier boundary conditions without any condition on the relative sizes of $\nu$ and $\al$ in dimension two, see \cite{busuioc_incompressible_2012}, and with the condition $\frac\nu\al$ bounded in dimension three, see \cite{busuioc_uniform_2016}.

In the case of the Dirichlet boundary conditions, convergence towards a solution of the Navier-Stokes equations when $\al\to0$ and $\nu>0$ is fixed was proved in \cite{busuioc_second_2016}, see also \cite{iftimie_remarques_2002-1}.

We would now like to know if the solutions of \eqref{secondgrade}--\eqref{Dir} converge to a solution of the incompressible Euler equation \eqref{euler}--\eqref{tan} when both $\al$ and $\nu$ converge to 0. The only result in that direction is given in \cite{lopes_filho_approximation_2015} where convergence towards a $H^3$ solution of the Euler equation is proved under the assumption that $\frac\nu\al$ is bounded.
\begin{theorem}[\cite{lopes_filho_approximation_2015}]\label{sgo}
Let $\ub_0\in H^3 $ be divergence free and tangent to the boundary. Assume that $u_0^{\al,\nu}$ verifies
\begin{itemize}
\item $u^{\al,\nu}_0\in H^3 $, $\dive u_0^{\al,\nu}=0$ and $u_0^{\al,\nu}\bigl|_{\partial\Om}=0$;
\item $\al^{\frac12}\nl2{\nabla u^{\al,\nu}_0}\to0$ and $\al^{\frac32}\|u^{\al,\nu}_0\|_{H^3}$ is bounded as $\al,\nu\to0$;
\item $u^{\al,\nu}_0\to \ub_0$ in $L^2$ as $\al,\nu\to0$. 
\end{itemize}
Assume moreover that  $\frac\nu\al$ is bounded.
Then the unique global $H^3$ solution of \eqref{secondgrade},\eqref{Dir} with initial data $u^{\al,\nu}_0$ converges in $L^\infty_{loc}([0,\infty);L^2 )$ when $\al\to0$ towards the unique global $H^3$ solution $\ub$ of \eqref{euler}--\eqref{tan} with initial data $\ub_0$.
\end{theorem}

We would now like to extend our result for the $\al$--Euler equations to the second grade fluids, proving convergence of \eqref{secondgrade}--\eqref{Dir} towards solutions of \eqref{euler}--\eqref{tan} with $L^p$ vorticity on multiply-connected domains. Our convergence result is based on $L^p$ estimates uniform in $\al$ for the vorticity $q$. Let us remark right away that such estimates can't hold true when $\al$ and $\nu$ are of the same size. More precisely, we have the following observation.
\begin{proposition} \label{unbound}
Under the hypotheses of Theorem \ref{sgo} assume in addition that both $\frac\nu\al$ and $\int_\Om q^{\al,\nu}_0$ have non-zero limits when $\al\to0$. Then for any $r>1$ the vorticity $q^{\al,\nu}$ is unbounded in $L^r_{loc}((0,\infty)\times\overline\Om)$. 
\end{proposition}
\begin{proof}
Let us apply the curl operator to \eqref{secondgrade}. We find that $q^{\al,\nu}=\curl(u^{\al,\nu}-\al\Delta u^{\al,\nu})$ verifies the following PDE:
\begin{equation}\label{equq}
\dt q^{\al,\nu}-\nu\curl\Delta u^{\al,\nu}+u^{\al,\nu}\cdot\nabla q^{\al,\nu}=0.  
\end{equation}
Integrating in space yields
\begin{equation*}
\dt\int_\Om q^{\al,\nu}-\nu \int_\Om\curl\Delta u^{\al,\nu}+  \int_\Om \dive(u^{\al,\nu}q^{\al,\nu})=0. 
\end{equation*}
Because $u^{\al,\nu}$ vanishes on the boundary, the Stokes formula implies that the last term on the left-hand side vanishes. For the same reason we have that $\int_\Om\curl u^{\al,\nu}=0$. We infer that
\begin{equation*}
 \int_\Om\curl\Delta u^{\al,\nu}=-\frac1\al\int_\Om q^{\al,\nu}. 
\end{equation*}
We deduce that
\begin{equation*}
\dt\int_\Om q^{\al,\nu}+\frac\nu\al  \int_\Om q^{\al,\nu}=0 
\end{equation*}
so
\begin{equation}\label{noncons}
\int_\Om q^{\al,\nu}(t)=e^{-\frac\nu\al t}\int_\Om q^{\al,\nu}_0.  
\end{equation}

By hypothesis, there exist $\ell_1,\ell_2\not =0$ such that
\begin{equation*}
\frac\nu\al\to\ell_1\quad\text{and}\quad  \int_\Om q^{\al,\nu}_0\to\ell_2.
\end{equation*}
Relation \eqref{noncons} implies that
\begin{equation}\label{limq}
\int_\Om q^{\al,\nu}(t)\to e^{-t\ell_1}\ell_2.  
\end{equation}

Now let us assume by absurd that $q^{\al,\nu}$ is bounded in $L^r_{loc}((0,\infty)\times\overline\Om)$ for some $r>1$. Then there is a subsequence of   $q^{\al,\nu}$, also denoted by $q^{\al,\nu}$, which converges  to some $\qb$ weakly in $L^r_{loc}((0,\infty)\times\overline\Om)$. Then $\int_\Om q^{\al,\nu}\to\int_\Om\qb$ weakly in $L^r_{loc}((0,\infty))$. In view of \eqref{limq} we infer that
\begin{equation}\label{intq}
  \int_\Om\qb(t,x)\,dx=e^{-t\ell_1}\ell_2
\end{equation}
almost everywhere in time.

But $u^{\al,\nu}\to\ub$ so $\al\curl\Delta u^{\al,\nu}\to0$ in the sense of distributions. Consequently $q^{\al,\nu}=\curl u^{\al,\nu}-\al\curl\Delta u^{\al,\nu}\to \curl\ub$ in the sense of distributions. By uniqueness of limits in the sense of distributions, we infer that $\qb=\curl\ub$. This is a contradiction because for a solution of the Euler equation the integral of vorticity is conserved in time while \eqref{intq} implies that the integral of $\qb$ is not constant in time. This completes the proof.
\end{proof}

Proposition \ref{unbound} shows that we can't hope to adapt our approach to second grade fluids if $\nu$ and $\al$ are of the same size. But if $\nu$ is slightly smaller in size than $\al$ then we can prove convergence to the Euler equations.

\begin{theorem}\label{secondthm}
Let $\Om$ be a smooth bounded domain of $\R^2$ and $1<p<\infty$. Assume that $\nu\leq \al^{1+\ep}$ for some $\ep>0$ independent of $\al$.
Let $\ub_0\in W^{1,p} $ be divergence free and tangent to the boundary. Let $u_0^{\al,\nu}$ be such that
\begin{itemize}
\item $u_0^{\al,\nu}\in W^{3,p} $, $\dive u_0^{\al,\nu}=0$ and $u_0^{\al,\nu}\bigl|_{\partial\Om}=0$;
\item $\nl2{u_0^{\al,\nu}}$, $\al^{\frac12}\nl2{\nabla u_0^{\al,\nu}}$ and $\nl p{q_0^{\al,\nu}}$ are bounded independently of $\al$;
\item $u_0^{\al,\nu}\to \ub_0$ in $L^2$ as $\al\to0$. 
\end{itemize}
Then there exists a global solution $\ual\in L^\infty(\R_+;W^{3,p} )$ of \eqref{secondgrade}--\eqref{Dir}. Moreover, there exists a subsequence of solutions $u^{\al_k,\nu_k}$ and a global solution $\ub$ of the Euler equations \eqref{euler}--\eqref{tan} with initial data $\ub_0$ such that $u^{\al_k,\nu_k}\to \ub$ in $L^\infty_{loc}([0,\infty);L^2 )$. In addition, if $\ep\geq\frac12$ then the limit solution has $L^p$ vorticity: $\curl\ub\in L^\infty_{loc}([0,\infty);L^p )$. If $\ep<\frac12$ then the limit solution has $L^r$ vorticity, $\curl\ub\in L^\infty_{loc}([0,\infty);L^r)$, for any $r$ verifying $1< r\leq p$ and $r<\frac1{1-2\ep}$.
\end{theorem}
\begin{remark}
The conclusion of Theorem \ref{secondthm} is slightly better than stated in the sense that we actually obtain convergence in $L^\infty_{loc}([0,\infty);W^{s,r} )$ for all $s<\frac1{r}$ and, by Sobolev embeddings, in $L^\infty_{loc}([0,\infty);H^{s'} )$ for all $s'<\min(1-\frac1{r},\frac12)$.  Here $r$ is either $p$ if $\ep\geq\frac12$, or any real number verifying $1< r\leq p$ and $r<\frac1{1-2\ep}$ if $\ep<\frac12$.
A second remark is that Theorem \ref{secondthm} is somewhat weaker than Theorem \ref{mainthm} in the sense that the limit solution does not always have vorticity in $L^p$ as we would expect from the initial vorticity belonging to $L^p$.
\end{remark}

The remainder of this section is devoted to the proof of Theorem \ref{secondthm}. We will not give all the details as the proof is very similar to the proof of Theorem \ref{mainthm}. We will only underline the differences. For clarity reasons we drop the superscript $^{\al,\nu}$ on the various quantities.

If we analyze the proof we gave for the $\al$--Euler equations, we realize that there are three main ingredients that need to be checked in the case of second grade fluids:
\begin{itemize}
\item $H^1_\al$ estimates for the velocity $u$;
\item Estimates for the circulations of $v$ on each $\Gamma_i$;
\item $L^p$ estimates for $q$.
\end{itemize}

The $H^1_\al$ estimates for $u$ go through easily. Indeed, if we multiply \eqref{secondgrade} by $u$, integrate in the $x$ variable and do some integrations by parts using that $u$ vanishes on the boundary we obtain that
\begin{equation*}
\dt\nhua u^2+2\nu\nl2{\nabla u}^2=0.  
\end{equation*}
So the $H^1_\al$ norm of $u$ decreases.

The circulations of $v$ on each $\Gamma_i$ are not conserved anymore but can nevertheless be computed and shown to be decreasing. More precisely, let
\begin{equation*}
\gamma_i(t)=  \int_{\Gamma_i}v(t)\cdot n^\perp.
\end{equation*}
We have the following result.
\begin{lemma}%\label{circulation2}
Let $u$ be a sufficiently smooth solution of \eqref{secondgrade}. Then for every $i\in\{1,\dots,n\}$ the circulation of $v$ on $\Gamma_i$ is given by
\begin{equation*}
\gamma_i(t)=\gamma_i(0)e^{-\frac\nu\al t}.  
\end{equation*}

\end{lemma}
\begin{proof}
We proceed like in the proof of Lemma \ref{circulation} by multiplying \eqref{secondgrade}  by $n^\perp$ and integrating on $\Gamma_i$. We get 
\begin{equation*}
\dt\int_{\Gamma_i}v\cdot n^\perp-\nu\int_{\Gamma_i}\Delta u\cdot n^\perp =0.
\end{equation*}
Since $u$ vanishes on $\Gamma_i$ we have that $\Delta u=-\frac1\al(u-\al\Delta u)=-\frac v\al$ on $\Gamma_i$. We infer that
\begin{equation*}
\dt\int_{\Gamma_i}v\cdot n^\perp+\frac\nu\al\int_{\Gamma_i}v\cdot n^\perp =0.
\end{equation*}
Consequently
\begin{equation*}
\int_{\Gamma_i}v(t)\cdot n^\perp = e^{-\frac\nu\al t} \int_{\Gamma_i}v(0)\cdot n^\perp .
\end{equation*}
\end{proof}

It remains to see if we can get $L^p$ estimates for $q$ and this is where the trouble lies. We can prove the following:
\begin{lemma}
For any $\delta>0$ and $1<r\leq p$ there exists a constant $C$ depending solely on $\delta$, $r$ and $\Om$ such that
\begin{equation}
  \label{eq:14}
\nl {r}{q(t)}\leq (\nl {r}{q_0}+\nhua{u_0})e^{Ct\nu\al^{-\frac32+\frac1{2{r}}-{\delta}}}.   
\end{equation}
\end{lemma}
\begin{proof}
We multiply \eqref{equq} by $q|q|^{{r}-2}$ and integrate in space to obtain that
\begin{equation*}
%  \label{eq:4}
\int_\Om \dt q \,q|q|^{{r}-2}-\nu\int_\Om \curl\Delta u\,q|q|^{{r}-2}+\int_\Om u\cdot\nabla q\,q|q|^{{r}-2}=0.    
\end{equation*}
We observe that $\dt q \,q|q|^{{r}-2}=\frac1{{r}}\dt|q|^{{r}}$ and  $\nabla q \,q|q|^{{r}-2}=\frac1{{r}}\nabla|q|^{{r}}$. Let $\om=\curl u$. Making an integration by parts and recalling that $\curl\Delta u=\frac{\om-q}\al$  we deduce that
\begin{equation*}
%  \label{eq:5} 
\frac1{{r}}\dt\nl{{r}}q^{{r}}+\frac\nu\al\nl{{r}}q^{{r}}=\frac\nu\al\int_\Om\om\,q|q|^{{r}-2}\leq \frac\nu\al\nl{{r}}\om\nl{{r}}q^{{r}-1}
\end{equation*}
so
\begin{equation}
  \label{eq:6} \dt\nl{{r}}q+\frac\nu\al\nl{{r}}q\leq \frac\nu\al\nl{{r}}\om.  
\end{equation}

Using \eqref{eq:8u} and recalling that the $H^1_\al$ norm of $u$ is decreasing we bound
\begin{equation*}
\nl{{r}}\om\leq \sob1{r}u\leq C\al^{-\frac12+\frac1{2{r}}-{\delta}}(\nl {r}q+\nhua u)\leq C\al^{-\frac12+\frac1{2{r}}-{\delta}}(\nl {r}q+\nhua {u_0}).
\end{equation*}
Using this in \eqref{eq:6} implies that
\begin{equation*}
%  \label{eq:13}
\dt\nl {r}q\leq C\nu\al^{-\frac32+\frac1{2{r}}-{\delta}}(\nl {r}q+\nhua{u_0}). 
\end{equation*}
The Gronwall inequality completes the proof of the lemma.
\end{proof}

Recalling that we assumed $\nu\leq\al^{1+\ep}$ we deduce from \eqref{eq:14} the following bound
\begin{equation}\label{finalbound}
\nl {r}{q(t)}\leq \bigl(\nl {r}{q_0}+\nhua{u_0}\bigr)e^{Ct\al^{-\frac12+\frac1{2{r}}+\ep-{\delta}}}  
\leq \bigl(\nl {r}{q_0}+\nhua{u_0}\bigr)e^{Ct}  
\end{equation}
provided that
\begin{equation}\label{condition}
  \ep+\frac1{2r}\geq \frac12+\delta.
\end{equation}

We consider two cases. 

If $\ep\geq\frac12$ then we choose $r=p$ and $\delta=\ep+\frac1{2p}-\frac12$ and we use \eqref{finalbound} to deduce that $q$ is bounded in the space $L^\infty_{loc}([0,\infty);L^{p})$ independently of $\al$ and $\nu$. Given the boundedness of the $L^{p}$ norm of $q$, the decay of the $H^1_\al$ norm of $u^{\al,\nu}$ and the explicit formula for the circulations $\gamma_i(t)$ one can argue as in the case of the $\al$--Euler equations and pass to the limit in \eqref{secondgrade} towards a solution of the Euler equation. Indeed, the additional term $-\nu\Delta u^{\al,\nu}$ is linear and goes to 0. The limit solution have vorticity in $L^\infty_{loc}([0,\infty);L^{p})$ and the convergence holds true in the space $L^\infty_{loc}([0,\infty);W^{s,p})$ for any $s<\frac1p$ (with the strong topology).

If $\ep<\frac12$ then we choose an $r$ such that $1<r\leq p$ and $r<\frac1{1-2\ep}$. The condition \eqref{condition} is verified for $\delta=\ep+\frac1{2r}-\frac12>0$. We obtain then from \eqref{finalbound} that $q$ is bounded in the space $L^\infty_{loc}([0,\infty);L^{r})$ independently of $\al$ and $\nu$. In this case we obtain convergence in $L^\infty_{loc}([0,\infty);W^{s,r})$ for any $s<\frac1r$ and the limit solution have vorticity in $L^\infty_{loc}([0,\infty);L^{r})$. 

The proof of Theorem \ref{secondthm} is completed.

\subsection*{Acknowledgments}   D.I. has been partially funded by the ANR project Dyficolti ANR-13-BS01-0003-01 and by the LABEX MILYON (ANR-10-LABX-0070) of Universit\'e de Lyon, within the program "Investissements d'Avenir" (ANR-11-IDEX-0007) operated by the French National Research Agency (ANR).

% \bibliographystyle{myabbrven}
% \bibliography{../OTHER/MACROS/zotero}

\begin{thebibliography}{10}

\bibitem{auchmuty_bounds_2016}
G.~Auchmuty.
\newblock Bounds and regularity of {Solutions} of {Planar} {Div}-curl
  {Problems}.
\newblock 2016.
\newblock arXiv: 1610.06827.

\bibitem{auchmuty_$l^2$_2001}
G.~Auchmuty and J.~C. Alexander.
\newblock ${L}^2$ {Well}-{Posedness} of {Planar} {Div}-{Curl} {Systems}.
\newblock {\em Archive for Rational Mechanics and Analysis}, 160(2):91--134,
  2001.

\bibitem{busuioc_second_2016}
A.~V. Busuioc.
\newblock From second grade fluids to the {Navier}-{Stokes} equations.
\newblock 2016.
\newblock arXiv: 1607.06689.

\bibitem{busuioc_incompressible_2012}
A.~V. Busuioc, D.~Iftimie, M.~C. Lopes~Filho and H.~J. Nussenzveig~Lopes.
\newblock Incompressible {Euler} as a limit of complex fluid models with
  {Navier} boundary conditions.
\newblock {\em Journal of Differential Equations}, 252(1):624--640, 2012.

\bibitem{busuioc_uniform_2016}
A.~V. Busuioc, D.~Iftimie, M.~C. Lopes~Filho and H.~J. Nussenzveig~Lopes.
\newblock Uniform time of existence for the alpha {Euler} equations.
\newblock {\em Journal of Functional Analysis}, 271(5):1341--1375, 2016.

\bibitem{cattabriga_su_1961}
L.~Cattabriga.
\newblock Su un problema al contorno relativo al sistema di equazioni di
  {Stokes}.
\newblock {\em Rendiconti del Seminario Matematico della Universit{\`a} di
  Padova}, 31:308--340, 1961.

\bibitem{cioranescu_weak_1997}
D.~Cioranescu and V.~Girault.
\newblock Weak and classical solutions of a family of second grade fluids.
\newblock {\em International Journal of Non-Linear Mechanics}, 32(2):317--335,
  1997.

\bibitem{cioranescu_mechanics_2016}
D.~Cioranescu, V.~Girault and K.~R. Rajagopal.
\newblock {\em Mechanics and mathematics of fluids of the differential type},
  volume~35 of {\em Advances in {Mechanics} and {Mathematics}}.
\newblock Springer, 2016.

\bibitem{cioranescu_existence_1984}
D.~Cioranescu and E.~H. Ouazar.
\newblock Existence and uniqueness for fluids of second grade.
\newblock In {\em Nonlinear partial differential equations and their
  applications. {Coll{\`e}ge} de {France} seminar, {Vol}. {VI} ({Paris},
  1982/1983)}, pages 178--197. Pitman, Boston, MA, 1984.

\bibitem{constantin_navier-stokes_1988}
P.~Constantin and C.~Foias.
\newblock {\em Navier-{Stokes} equations}.
\newblock Chicago {Lectures} in {Mathematics}. University of Chicago Press,
  Chicago, IL, 1988.

\bibitem{dunn_thermodynamics_1974}
J.~E. Dunn and R.~L. Fosdick.
\newblock Thermodynamics, stability, and boundedness of fluids of complexity
  $2$ and fluids of second grade.
\newblock {\em Arch. Rational Mech. Anal.}, 56:191--252, 1974.

\bibitem{fujiwara_asymptotic_1969}
D.~Fujiwara.
\newblock On the asymptotic behaviour of the {Green} operators for elliptic
  boundary problems and the pure imaginary powers of some second order
  operators.
\newblock {\em JMSJ}, 21(4), 1969.

\bibitem{galdi_existence_1993}
G.~P. Galdi, M.~Grobbelaar-Van~Dalsen and N.~Sauer.
\newblock Existence and uniqueness of classical solutions of the equations of
  motion for second-grade fluids.
\newblock {\em Archive for Rational Mechanics and Analysis}, 124(3):221--237,
  1993.

\bibitem{galdi_further_1994}
G.~P. Galdi and A.~Sequeira.
\newblock Further existence results for classical solutions of the equations of
  a second-grade fluid.
\newblock {\em Archive for Rational Mechanics and Analysis}, 128(4):297--312,
  1994.

\bibitem{giga_analyticity_1981}
Y.~Giga.
\newblock Analyticity of the semigroup generated by the {Stokes} operator in
  ${L}^r$ spaces.
\newblock {\em Mathematische Zeitschrift}, 178(3):297--329, 1981.

\bibitem{giga_domains_1985}
Y.~Giga.
\newblock Domains of fractional powers of the {Stokes} operator in ${L}_r$
  spaces.
\newblock {\em Archive for Rational Mechanics and Analysis}, 89(3):251--265,
  1985.

\bibitem{holm_euler-poincare_1998-1}
D.~D. Holm, J.~E. Marsden and T.~S. Ratiu.
\newblock Euler-{Poincar{\'e}} models of ideal fluids with nonlinear
  dispersion.
\newblock {\em Phys. Rev. Lett.}, 80(19):4173--4176, 1998.

\bibitem{iftimie_remarques_2002-1}
D.~Iftimie.
\newblock Remarques sur la limite $\alpha\to0$ pour les fluides de grade 2.
\newblock In {\em Nonlinear partial differential equations and their
  applications. {Coll{\`e}ge} de {France} {Seminar}, {Vol}. {XIV} ({Paris},
  1997/1998)}, volume~31 of {\em Stud. {Math}. {Appl}.}, pages 457--468.
  North-Holland, Amsterdam, 2002.

\bibitem{kato_remarks_1984}
T.~Kato.
\newblock Remarks on zero viscosity limit for nonstationary {Navier}-{Stokes}
  flows with boundary.
\newblock In {\em Seminar on nonlinear partial differential equations
  ({Berkeley}, {Calif}., 1983)}, volume~2 of {\em Math. {Sci}. {Res}. {Inst}.
  {Publ}.}, pages 85--98. Springer, New York, 1984.

\bibitem{lopes_filho_vortex_2007}
M.~C. Lopes~Filho.
\newblock Vortex dynamics in a two-dimensional domain with holes and the small
  obstacle limit.
\newblock {\em SIAM Journal on Mathematical Analysis}, 39(2):422--436
  (electronic), 2007.

\bibitem{lopes_filho_approximation_2015}
M.~C. Lopes~Filho, H.~J. Nussenzveig~Lopes, E.~S. Titi and A.~Zang.
\newblock Approximation of 2D {Euler} {Equations} by the {Second}-{Grade}
  {Fluid} {Equations} with {Dirichlet} {Boundary} {Conditions}.
\newblock {\em Journal of Mathematical Fluid Mechanics}, 17(2):327--340, 2015.

\bibitem{lopes_filho_convergence_2015}
M.~C. Lopes~Filho, H.~J. Nussenzveig~Lopes, E.~S. Titi and A.~Zang.
\newblock Convergence of the 2D {Euler}-$\alpha$ to {Euler} equations in the
  {Dirichlet} case: {Indifference} to boundary layers.
\newblock {\em Physica D: Nonlinear Phenomena}, 292{\textendash}293:51--61,
  2015.

\bibitem{oliver_vortex_2001}
M.~Oliver and S.~Shkoller.
\newblock The vortex blob method as a second-grade non-{Newtonian} fluid.
\newblock {\em Communications in Partial Differential Equations},
  26(1-2):295--314, 2001.

\bibitem{taylor_partial_1997}
M.~E. Taylor.
\newblock {\em Partial differential equations. {III}}.
\newblock Springer-Verlag, New York, 1997.

\end{thebibliography}

\bigskip

\begin{description}
\item[A. V. Busuioc:] Université de Lyon, Université de Saint-Etienne  --
CNRS UMR 5208 Institut Camille Jordan --
Faculté des Sciences --
23 rue Docteur Paul Michelon --
42023 Saint-Etienne Cedex 2, France.\\
Email: \texttt{valentina.busuioc@univ-st-etienne.fr}
\item[D. Iftimie:] Universit\'e de Lyon, Universit\'e Lyon 1 --
CNRS UMR 5208 Institut Camille Jordan --
43 bd. du 11 Novembre 1918 --
Villeurbanne Cedex F-69622, France.\\
Email: \texttt{iftimie@math.univ-lyon1.fr}
\end{description}

\end{document}